\documentclass[final]{siamltex}
\usepackage{amssymb, amsmath, mathrsfs}

\usepackage{newsymbol}
\let\ge       \undefined
\let\le       \undefined
\newsymbol\le        1336  \let\leq\le
\newsymbol\ge        133E  \let\geq\ge
\renewcommand{\le}{\leq}
\renewcommand{\ge}{\geq}

\newtheorem{remark}[theorem]{Remark}
\newtheorem{example}[theorem]{Example}
\newtheorem{claim}[theorem]{Claim}

\newcommand{\R}{{\mathbb{R}}}

\newcommand{\E}{{\mathbb{E}}}
\renewcommand{\P}{{\mathbb{P}}}

\newcommand{\F}{\mathscr{F}}

\newcommand{\g}{\gamma}
\newcommand{\s}{^{\ast}}

\newcommand{\om}{\omega}
\newcommand{\Om}{\Omega}

\newcommand{\n}{\Vert}
\newcommand{\embed}{\hookrightarrow}

\newcommand{\lb}{\langle}
\newcommand{\rb}{\rangle}
\renewcommand{\a}{\alpha}
\renewcommand{\b}{\beta}
\renewcommand{\d}{\delta}
\newcommand{\calL}{\mathscr{L}}
\newcommand{\limn}{\lim_{n\to\infty}}

\newcommand{\limj}{\lim_{j\to\infty}}
\renewcommand{\H}{\mathscr{H}}

\newcommand{\inv}[1]{\frac{1}{#1}}
\newcommand{\tinv}[1]{\tfrac{1}{#1}}

\newcommand{\beq}{\begin{equation}}
\newcommand{\eeq}{\end{equation}}
\newcommand{\bal}{\begin{aligned}}
\newcommand{\eal}{\end{aligned}}
\newcommand{\ben}{\begin{enumerate}}
\newcommand{\een}{\end{enumerate}}
\newcommand{\bit}{\begin{itemize}}
\newcommand{\eit}{\end{itemize}}

\title{Convergence rates of the splitting scheme for parabolic linear stochastic
Cauchy problems}

\author{Sonja Cox\!\!\and\,Jan van Neerven\thanks{S.G.Cox@tudelft.nl, J.M.A.M.vanNeerven@tudelft.nl, Delft Institute of Applied Mathematics; Delft University of Technology; P.O. Box 5031; 2600 GA Delft; The Netherlands. The second named author gratefully acknowledges support by VICI subsidy
639.033.604 of the Netherlands Organisation for Scientific Research (NWO).}}

\begin{document}

\maketitle

\begin{abstract}
We study the splitting scheme associated with the linear stochastic Cauchy
problem
$$\begin{aligned}
dU(t)& = AU(t)\,dt + dW(t),\quad t\in [0,T],  \\
U(0) & = x,
\end{aligned}
$$
Here $A$ is the generator of an analytic $C_0$-semigroup $S=\{S(t)\}_{t\ge 0}$
on a Banach space $E$ and 
$W=\{W(t)\}_{t\ge 0}$ is a Brownian motion with values in a fractional domain
space $E_\b$ associated with $A$.
We prove that if
$\a,\b,\g,\theta\ge 0$ are such that $\g+\theta<1$ and $(\a-\b+\theta)^+
+\g<\frac12$,
then the approximate solutions $U^{(n)}$ converge to the solution $U$ in
the space $C^\g([0,T];E_{\alpha})$, both in $L^p$-means and almost surely, with
rate $1/n^{\theta}$.
\end{abstract}

\begin{keywords}
Splitting scheme, stochastic evolution equations, 
analytic semigroups, $\g$-radon\-ifying operators, $\g$-boundedness, stochastic
convolutions, Lie-Trotter product formula.
\end{keywords}

\begin{AMS}
Primary: 35R15, 60H15; Secondary: 47D06, 60J35
\end{AMS}

\pagestyle{myheadings}
\thispagestyle{plain}
\markboth{S. COX AND J.M.A.M. VAN NEERVEN}{CONVERGENCE RATES OF SPLITTING-UP SCHEME FOR SACP}

\maketitle

\section{Introduction}

We are concerned with the convergence of the splitting scheme for the stochastic
linear Cauchy problem
\beq\label{sACP}\tag{SCP$_x$}
\left\{\begin{aligned}
dU(t) & = AU(t)\,dt + dW(t),\quad t\in [0,T],  \\
U(0) & = x,
\end{aligned}
\right.
\eeq
were $A$ is the generator of a $C_0$-semigroup $S=\{S(t)\}_{t\ge 0}$ on a 
real Banach space $E$,
$W=\{W(t)\}_{t\ge 0}$ is an $E$-valued Brownian motion on a probability space
$(\Om,\P)$, and $x\in E$ is an initial value which is kept fixed throughout the
paper. The concept of the scheme is to alternately add an increment of the
Brownian motion $W$ and run the semigroup $S$ on a time interval of equal
length. Taking time steps $\Delta t^{(n)} := T/n$ and writing $t_j^{(n)} :=
jT/n$ and $\Delta W_j^{(n)}:= W(t_j^{(n)}) - W(t_{j-1}^{(n)})$, this generates a
finite sequence $\{U_x^{(n)}(t_j^{(n)})\}_{j=0}^n$ defined by
$$ 
\bal
U_x^{(n)}(t_0^{(n)}) &:= x, \\  
U_x^{(n)}(t_j^{(n)}) &:= S(\Delta t^{(n)})(U_x^{(n)}(t_{j-1}^{(n)})
+ \Delta W_j^{(n)}), \quad j=1,\dots,n.
\eal
$$
We have the explicit formula
$$
U_x^{(n)}(t_j^{(n)}) 
= S(t_j^{(n)})x + \sum_{i=1}^j S(t_{j-i+1}^{(n)})\Delta W_i^{(n)}, \quad
j=0,\dots,n. 
$$
Assuming the existence of a unique solution $U_x$ of the problem \eqref{sACP}
(see Proposition \ref{prop:RPhi} below),
we may ask for conditions ensuring the convergence of $U_x^{(n)}(T)$ to $U_x(T)$
in $L^p(\Om;E)$ 
for some (all) $1\le p<\infty$ or even in $E$ almost surely.
In order to describe our approach we start by noting that 
each $U_x^{(n)}(t_j^{(n)})$ can be represented as a stochastic integral
of the discretised function
$$ S^{(n)}(t):= \sum_{j=0}^n 1_{I_j^{(n)}}\otimes S(t_j^{(n)}), \quad t\in
[0,T],$$
where $I_0^{(n)} = \{0\}$ and $I_j^{(n)} = (t_{j-1}^{(n)}, t_j^{(n)}]$ for
$j=1,\dots,n$. Indeed, defining the stochastic integral of a step function in
the obvious way, 
we have
\begin{equation}\label{eq:SI1}
U_x^{(n)}(t_j^{(n)}) = S^{(n)}(t_j^{(n)})x + \int_0^{t_j^{(n)}}
S^{(n)}(t_j^{(n)} -s)\,dW(s),\quad j=0,\dots,n.
\end{equation}
On the other hand, the exact solution of \eqref{sACP}, if it exists, is given by
the stochastic convolution integral
\beq\label{eq:SI2} 
U_x(t):= S(t)x + \int_0^{t} S(t-s)\,dW(s), \quad t\in [0,T].
\eeq
For the precise definition of the stochastic integral we refer to Section
\ref{sec:main}. 
Comparing \eqref{eq:SI1} and \eqref{eq:SI2} we see that the problem of
convergence of the splitting scheme is really a problem of convergence of
`Riemann sums' for stochastic integrals. Let us henceforth put
$$
U_x^{(n)}(t)  := S^{(n)}(t)x + \int_0^{t} S^{(n)}(t-s)\,dW(s), \quad t\in [0,T].
$$
The second formula interpolates the data in the identity \eqref{eq:SI1}
in a way that makes them easily accessible with continuous time techniques;
other possible interpolations, such as piecewise linear interpolation, do not
have this advantage. Needless to say, in Theorems \ref{thm1} and \ref{thm2b}
below
we are primarily interested in what happens at the time points $t=t_j^{(n)}$.
From $S^{(n)}(t_j^{(n)})x = S(t_j^{(n)})x$ we see that
$$ U_x^{(n)}(t_j^{(n)}) - U_x(t_j^{(n)}) = U_0^{(n)}(t_j^{(n)}) -
U_0(t_j^{(n)})$$ for all $x\in E$, and therefore it suffices to analyse
convergence of the splitting scheme with initial value $0$. In what follows, in
order to simplify notations we shall write $U(t) := U_0(t)$ and $U^{(n)}(t):=
U_0^{(n)}(t)$.  

Our first result extends and simplifies previous work by K\"uhne\-mund and the
second-named author \cite[Theorems 4.3 and 5.2]{KueNee}.

\begin{theorem}\label{thm1}
Each of the conditions {\rm(a)} and {\rm(b)} below guarantees that the problem
\beq\label{sACP0}\tag{SCP$_0$}
\left\{\begin{aligned}
dU(t) & = AU(t)\,dt + dW(t),\quad t\in [0,T],  \\
U(0) & = 0,
\end{aligned}
\right.
\eeq
admits a unique 
solution $U= \{U(t)\}_{t\in [0,T]}$ which satisfies 
$$\limn \Big(\sup_{t\in [0,T]} \E \n U^{(n)}(t) - U(t)\n^p\Big) = 0$$
for all $1\le p<\infty$:
\ben
\item[(a)] $E$ has type $2$;
\item[(b)] $S$ restricts to a $C_0$-semigroup on the reproducing kernel Hilbert
space associated with $W$.
\een
\end{theorem}

The class of spaces satisfying condition (a) includes all Hilbert spaces and the
spaces $L^p(\mu)$ for $2\le p<\infty$. It follows from the results in
\cite{MaaNee} that condition (b) is satisfied if the transition semigroup
associated with the solution process is analytic. 

The main result of this article, Theorem \ref{thm2b} below, concerns actual
convergence rates for the splitting scheme in the case that the semigroup $S$ is
{\em analytic} on $E$.  The convergence is considered in suitable H\"older norms
in space and time, with explicit bounds for the convergence rate. 

We denote by $E_\a$ the fractional power space of exponent $\a\ge 0$ associated
with $A$ (see Section \ref{sec:analytic} for more details).

\begin{theorem}\label{thm2b}
Suppose that the semigroup $S$ is analytic on $E$ and that $W$ is a Brownian
motion in $E_\b$ for some $\b\ge 0$. 
Then the problem \eqref{sACP0} admits a unique 
solution $U= \{U(t)\}_{t\in [0,T]}$, and for all
$\a,\g,\theta\ge 0$ such that $\g+\theta<1$ and $(\a-\b+\theta)^+ + \g <
\frac12$ one has the estimate
$$\begin{aligned}
\big(\E\big\Vert U^{(n)} - U \big\Vert^p_{C^\g([0,T];E_{\alpha})}\big)^\frac1p
\lesssim \frac1{n^{\theta}},
\quad 1\le p<\infty,
\end{aligned}$$
with implied constant independent of $n\ge 1$. 
\end{theorem}
 
By a Borel-Cantelli argument, this result implies the almost sure convergence of
$U^{(n)}$ to $U$ in $C^\g([0,T];E_{\alpha})$ with the same rates.

The proof of Theorem \ref{thm2b} heavily relies
on the theory of $\g$-radonifying operators and $\g$-bounded\-ness techniques.
Standard techniques from stochastic analysis which are
commonly used in connection with the problems considered here, such as It\^o's
formula and the Burkholder-Davis-Gundy inequalities, are unavailable in the
present general framework (unless one makes additional assumptions on $E$, such
as martingale type $2$ or the UMD property). We also cannot use 
factorisation techniques (as introduced
by Da Prato, Kwapie\'n and Zabczyk \cite{DKZ}), the reason being that the
semigroup property on which this technique relies fails for the discretised
semigroup $S^{(n)}$.

\begin{example}
\rm Theorem \ref{thm2b} may be applied to second order elliptic operators 
of the form 
$$ Af(x) = \sum_{i,j=1}^d a_{ij}(x)\partial_{ij}f(x)
+  \sum_{i=1}^d b_{i}(x)\partial_i f(x) + c(x)f(x).
$$
Under minor regularity assumptions on the coefficients $a_{ij}=a_{ji}$, $b_i$ and $c$, 
such operators generate analytic semigroups on
$E=L^q(\R^d)$ with $1<p<\infty$ (see \cite[Chapter 3]{Lun}) and one
has $E_\a = H^{2\a,q}(\R^d)$ for all $0<\a<\frac12$. Applying Theorem
\ref{thm2b} (with $\b=0$), we obtain convergence of the splitting scheme in the space
$C^\g([0,T];H^{2\a,q}(\R^d))$ for any $\g\ge  0$ such that $0<\a+\g<\frac12$. By
the Sobolev embedding $\smash{ H^{2\a,q}(\R^d)\embed C_0^{2\a-{d}/q}(\R^d)}$
 \cite[Section 2.8]{Tri}, this implies the convergence of the splitting scheme
in the mixed H\"older space $C^\g([0,T];C_0^{2\a-{d}/q}(\R^d))$. As a
consequence, we obtain convergence in the mixed H\"older space
$$C^\g([0,T];C_0^{2\d}(\R^d)), \quad \g,\d\ge 0, \ \ \g+\d<\tfrac12,$$
with rate $1/n^\theta$ for any $\theta<\tfrac12-\g-\d$; this rate improves when
the noise is more regular. 
Similar results can be obtained for elliptic operators on smooth domains 
$D\subseteq\R^d$
subject to various types of boundary conditions (as long as they generate an 
analytic semigroup on $L^q(D)$).
\end{example}

For semi-linear (Stratonovich type) SPDEs governed by second order elliptic
operators on $\R^d$ and driven by multiplicative noise, convergence in $E=
L^2(\R^d)$ of splitting schemes like the one considered here has been proved by
various authors \cite{BGR1, BGR2, FloLeG, GyoKry, Nag}. Using techniques from
PDE and stochastic analysis it is shown by Gy\"ongy and Krylov \cite{GyoKry}
that, with respect to the norm of $E=L^2(\R^d)$, for finite-dimensional noise
and with sufficiently smooth coefficients one obtains the maximal estimate
$$\big(\E\sup_{t\in [0,T]}\n U^{(n)}- U\n_{L^2(\R^d)}^p\big)^\frac1p \lesssim
\frac1{n}, \quad 1\le p<\infty.$$
Our result is valid in the full scale of spaces $L^q(\R^d)$ and
infinite-dimensional noise, with a rate which (for smooth enough noise) is only
slightly worse that $1/n$ and is independent of $1\le q<\infty$. More precisely, 
for $\b\ge \frac12+\a$ and taking $\g=0$ we obtain uniform convergence with rate 
$1/n^\theta$ for any $0\le\theta<1$. In addition to
that we obtain H\"older regularity in both space and time. 
On the other hand, as we already mentioned, Gy\"ongy and Krylov
\cite{GyoKry} consider the semi-linear case and multiplicative noise.

The next example shows  that by working in suitable fractional extrapolation
spaces (this technique is explained in \cite{DNW}; see also \cite{Brz95,
Brz97}), the 
assumption that $W$ is a Brownian motion can be weakened to $W$ being a
cylindrical Brownian motion (see, e.g., \cite{NVW07a, NeeWei05a} for the
definition). 

\begin{example}\label{ex:stoch-heat-intro}
\rm The stochastic heat equation on the unit interval $[0,1]$ with Dirichlet
boundary conditions driven by space-time white noise can be put into the present
framework by taking for $E$ the extrapolation space $F_\rho$ with $F = L^q(0,1)$
and $\rho<-\frac14$. As we shall explain in Example \ref{ex:stoch-heat}, this
entails the convergence of the splitting scheme in the mixed H\"older space
$$C^\g([0,T];C_0^{2\d}[0,1]),\quad \g,\d\ge 0, \ \ \g+\d<\tfrac14,$$ 
with rate $1/n^\theta$ for any $\theta<\frac14-\g-\d$.
\end{example}

It is shown in \cite{DaGai01} that any approximation scheme for a one-dimensional
stochastic heat equation with additive space-time white noise which incorporates
the contributions of the noise only by means of the terms $\Delta W_k$,
$k=1,\ldots,n$, cannot have a convergence rate better than $1/n^{\frac{1}{4}}$. 
This shows that the exponent $\frac14$ in Example \ref{ex:stoch-heat-intro} 
is the best possible.

The field of
numerical approximation of stochastic partial differential equations (SPDEs) is
a very active one; an up-to-date overview of the available results can be found in
\cite{JenKloe09a}. In \cite{Gyo99} convergence rates are considered for various
approximations schemes in space and time of a quasi-linear parabolic SPDE driven by 
white noise. The authors obtain a convergence rate $1/n^{\frac{1}{4}}$ in $L^p$ for an
implicit Euler scheme. In \cite{PetSig05} convergence in probability 
is proved (without rates) for the same SPDE with state-dependent dispersion. Rates for
path-wise convergence are given for quasi-linear parabolic SPDEs 
in \cite{GyoMil09, Jen09, LordRou04}, albeit only for colored noise. It seems likely that the methods of this 
paper can be extended to the implicit Euler scheme and to semi-linear problems with 
multiplicative noise; we plan to address such extensions in a future paper.

The paper is organised as follows. 
Section \ref{sec:prelim} presents some preliminary material about spaces of
$\g$-radonifying operators. The proofs of Theorems \ref{thm1} and \ref{thm2b}
are presented in Sections \ref{sec:main} (Theorems \ref{thm:b}, \ref{thm:a}) and
\ref{sec:analytic} (Theorem \ref{thm:unif}), respectively. 

It is known that each of the conditions in Theorems \ref{thm1}   and
\ref{thm2b} implies that the 
solution process $U$ has continuous trajectories.
In the final Section \ref{sec:example} we present an example which shows that
without any additional assumptions on the space $E$ and/or the semigroup $S$ the
splitting scheme may fail to converge even if a solution $U$ with continuous
trajectories exists.

\section{Preliminaries}\label{sec:prelim}

Let $\{\g_j\}_{j\ge 1}$ be a sequence of independent standard Gaussian random
variables on a probability space $(\Om,\P)$, let $\H$ be a real Hilbert space
(later we shall take $\H = L^2(0,T;H)$, where $H$ is another real Hilbert space)
and $E$ a real Banach space. A bounded operator $R$ from $\H$ to $E$
is called {\em $\g$-summing} if 
$$ \n R \n_{\g_{\infty}(\H,E)}^2 := \sup_h \E \Big\n \sum_{j=1}^k \g_j R h_j \Big\n^2,$$
is finite, where the supremum is taken over all finite orthonormal systems 
$h = \{h_j\}_{j=1}^k$ in $\H$. It can be shown that $\n\cdot \n_{\g_{\infty}(\H,E)}$ is indeed a norm which turns the space of $\g$-summing operators into a Banach space.

The space $\g(\H,E)$ of $\g$-\emph{radonifying} operators is defined to be the closure of the finite rank operators under the norm $\n\cdot\n_{\g_{\infty}}$; it is a closed subspace of $\g_{\infty}(\H,E)$. A celebrated result of Kwapie\'n and Hoffmann-J{\o}rgensen \cite{HofJor, Kwa}
implies that  if $E$ does not contain a closed subspace isomorphic to $c_0$ then $\g(\H,E)=\g_\infty(\H,E)$.

Since convergence in $\g(\H,E)$ implies convergence in $\calL(\H,E)$, every
operator $R\in \g(\H,E)$, being the operator norm limit of a sequence of finite
rank operators from $\H$ to $E$,  is compact.

If $\H$ is separable with orthonormal basis $\{h_j\}_{j\ge 1}$, then an operator
$R:\H\to E$ is $\g$-radonifying if and only if the Gaussian sum
$ \sum_{j\ge 1} \g_j Rh_j $ converges in $L^2(\Om;E)$, and in this situation we
have
$$ \n R\n_{\g(\H,E)}^2 =  \E \Big\n\sum_{j\ge 1} \g_j Rh_j \Big\n^2.$$
The general case may be reduced to the separable case by observing that for any
$R\in \g(\H,E)$ there exists a separable closed subspace $\H_R$ of $\H$ such
that $R$ vanishes on the orthogonal complement $\H_R^\perp$. 

If $R\in \g(H,E)$ is given and $\{h_j\}_{j\ge 1}$ is an orthonormal basis for
$\H_R$, 
the sum $\sum_{j\ge 1} \g_j Rh_j$ defines a centred $E$-valued Gaussian random
variable. Its distribution
$\mu$ is a centred Gaussian Radon measure on $E$ whose covariance operator
equals $RR\s$. We will refer to $\mu$ as the Gaussian measure {\em associated
with} $R$.
In the reverse direction, if $Y$ is a centred $E$-valued Gaussian 
random variable with reproducing kernel
Hilbert space $\H$, then $\H$ is separable, the natural inclusion mapping
$i:\H\embed E$ is $\g$-radonifying, and we have
$$ \n i\n_{\g(\H,E)}^2 = \E \n Y\n^2.$$

Below we shall need the following simple continuity result.

\begin{proposition}\label{prop:g-cont} 
Let $(X,d)$ be a metric space and let $V:X\to \calL(E,F)$ be
strongly continuous. Then for all $R\in
\g(\H,E)$ the function $VR: X\to \g(\H,F)$, $$(VR)(\xi):= V(\xi)R, \qquad \xi\in
X,$$ is continuous.
\end{proposition}

\begin{proof} 
Suppose first that $R$ is a finite rank operator, say $R = \sum_{j=1}^k
h_j\otimes x_j$ with $\{h_j\}_{j=1}^k\in \H$ orthonormal and $\{x_j\}_{j=1}^k$ a
sequence in $E$.
Suppose that $\limn \xi_n = \xi$ in $X$. 
Then
$$\limn \n V(\xi_n)R - V(\xi)R\n_{\g(\H,F)}^2 = \limn\E \Big\n \sum_{j=1}^k \g_j
(V(\xi_n)-V(\xi))x_j\Big\n^2 = 0.$$ 
The general case follows from the density of the finite rank operators in
$\g(\H,E)$ and the norm estimate $\n V(\xi)R\n_{\g(\H,F)}\le \n V(\xi)\n \n
R\n_{\g(\H,E)}$.
\end{proof}

\section{Proof of Theorem \ref{thm1}}
\label{sec:main}

We start with a brief discussion of stochastic integrals of operator-valued
functions. Let $H$ be a Hilbert space and fix $T>0$. An {\em $H$-cylindrical Brownian motion}, indexed by $[0,T]$ and defined on a probability space $(\Omega,\F,\P)$, is a mapping $W_{H}:L^2(0,T;H)\rightarrow L^2(\Omega)$ with the following properties:
\bit
\item for all $h\in L^2(0,T;H)$ the random variable $W_{H}(h)$ is Gaussian;
\item for all $h_1,h_2\in L^2(0,T;H)$ we have $\E W_{H}(h_1)W_{H}(h_2)=\langle h_1,h_2\rangle$. 
\eit
Formally, an $H$-cylindrical Brownian motion can be thought of as a standard Brownian motion taking values in the Hilbert space $H$. One easily checks that $W_{H}$ is linear and that for all $h_1,\ldots,h_n \in L^2(0,T;H)$ the random variables $W_{H}(h_1),\ldots,W_{H}(h_n)$ are jointly Gaussian. These random variables are independent if and only if $h_1,\ldots,h_n$ are orthogonal in $H$. For further details see \cite[Section 3]{Nee-survey}.\par

A {\em finite rank step function} is function of the form $\sum_{n=1}^N
1_{(a_n,b_n]}\otimes B_n$ where each operator $B_n:H\to E$ is of finite rank.
The stochastic integral with respect to $W_H$ of such a function is defined by
setting
$$ \int_0^T 1_{(a,b]}\otimes (h\otimes x) \,dW_H:= W_H(1_{(a,b]}\otimes h)\otimes x$$
and this definition is extended by linearity.
A function $\Psi:(0,T)\to \calL(H,E)$ is said to be {\em stochastically
integrable} with respect to $W_H$ if there exists a sequence of finite rank step
functions $\Psi_n:(0,T)\to \calL(H,E)$ such that:
\ben
\item[(i)] for all $h\in H$ we have $\limn \Psi_n h = \Psi h$ in measure on
$(0,T)$;
\item[(ii)] the limit 
$ Y:= \limn \int_0^T \Psi_n\,dW_H$
exists in probability.
\een
In this situation we write $$Y = \int_0^T \Psi\,dW_H$$
and call $Y$ the {\em stochastic integral} of $\Psi$ with respect to $W_H$. 

As was shown in \cite{NeeWei05a}, for finite rank step functions $\Psi$ one has
the isometry
\beq\label{eq:Ito} \E \Big\n  \int_0^T \Psi \,dW_H\Big\n^2 = \n
R_{\Psi}\n_{\g(L^2(0,T;H),E)}^2,
\eeq
where $R_\Psi:L^2(0,T;H)\to E$ is the bounded operator represented by $\Psi$,
i.e., 
\beq\label{eq:RPhi} R_\Psi f = \int_0^T \Psi(t)f(t)\,dt, \quad f\in
L^2(0,T;H).\eeq
As a consequence, a function $\Psi:(0,T)\to \calL(H,E)$ is 
stochastically integrable on $(0,T)$ with respect to $W_H$ if and only if 
$\Psi\s x\s \in L^2(0,T;H)$ for all $x\s\in E\s$ and there exists an operator
$R_\Psi\in\g(L^2(0,T;H),E)$ such that
$$ \lb R_\Psi f,x\s\rb = \int_0^T [f(t),\Psi\s(t)x\s]\,dt, \quad x\s\in E\s.
$$
The isometry \eqref{eq:Ito} extends to this situation. The following simple
observation \cite[Lemma 2.1]{DNW} will be used frequently:

\begin{proposition}\label{prop:gR}
For all $g\in L^2(0,T)$ and $R\in\g(H,E)$ the function $gR: t\mapsto g(t)R$
belongs to 
$\g(L^2(0,T;H),E)$ and we have 
$$ \n g R\n_{\g(L^2(0,T;H),E)}= \n g\n_{L^2(0,T)}\n R\n_{\g(H,E)}.$$
\end{proposition}

For the remainder of this section we fix an $E$-valued Brownian motion $W =
\{W(t)\}_{t\ge 0}$ and $T>0$. Let $H$ be the reproducing kernel Hilbert space 
associated with the Gaussian random variable $W(1)$ and let $i:H\embed E$ be the natural inclusion mapping. Then $W$ induces an $H$-cylindrical Brownian motion $W_H$ by putting
\begin{equation}
\begin{aligned}\label{eq:cylBM}
W_H(f\otimes i\s x\s):=\int_{0}^{T} f\,d\langle W,x^*\rangle, \quad f\in L^2(0,T), \ x^*\in E^*.
\end{aligned}\end{equation}
This motivates us to call a function
$\Psi:(0,T)\to \calL(E)$ {\em stochastically integrable} with respect to $W$
if the function $\Psi\circ i:(0,T)\to \calL(H,E)$ is stochastically integrable with respect to $W_H$, in which case we put
$$
\int_0^T \Psi\,dW := \int_0^T (\Psi\circ i)\,dW_H.
$$
It is easy to check that for all $S\in\calL(E)$ the indicator function
$1_{(a,b]}\otimes S$ is stochastically integrable with respect to $W$ and 
$$ \int_0^T 1_{(a,b]}\otimes S\,dW = S(W(b)-W(a)).$$
This shows that the definition is consistent with \eqref{eq:SI1} and \eqref{eq:SI2}.

Now let 
$S = \{S(t)\}_{t\ge 0}$ denote a $C_0$-semigroup of bounded linear operators on
$E$, with generator $A$.
We will be interested in the case where the function to be integrated against
$W_H$ is one of the following:
$$
\Phi(t) := S(t)\circ i, \quad \Phi^{(n)}(t) := \sum_{j=1}^n
1_{I_j^{(n)}}(t)\otimes
[S(t_j^{(n)})\circ i],\quad t\in (0,T).
$$  
We may define bounded operators $R_{\Phi^{(n)}}$ and $R_\Phi$ from $L^2(0,T;H)$
to $E$ by the formula \eqref{eq:RPhi}. Being associated with $\g(H,E)$-valued
step functions, the operators $R_{\Phi^{(n)}}$ belong to
$\g(L^2(0,T;H),E)$ by Proposition \ref{prop:gR}. 
Concerning the question whether the operator
$R_{\Phi}$ is in $\g(L^2(0,T;H),E)$ we have the following result 
\cite[Theorem 7.1]{NeeWei05a}.

\begin{proposition}\label{prop:RPhi}
Let $\Phi(t) = S(t)\circ i$. The following assertions are equivalent:
\ben
\item[\rm(i)] the operator $R_{\Phi}$ belongs to $\g(L^2(0,T;H),E)$;
\item[\rm(ii)] the function  $\Phi$ is stochastically integrable on $(0,T)$ with
respect to $W_H$;
\item[\rm(iii)] for some (all) $x\in E$ the problem \eqref{sACP} 
admits a unique solution $U_x$.
\een
In this situation, for all $x\in E$ and $t\in [0,T]$ we have 
$$
U_x(t) = S(t)x + \int_0^t S(t-s)\,dW(s) = S(t)x+\int_0^t \Phi(t-s)\,dW_H(s)
$$
almost surely. 
\end{proposition}
In \cite{NeeWei05a} an example is presented showing even for
rank one Brownian motions $W$ in $E$ the equivalent conditions need not always
be satisfied for all $C_0$-semigroups $S$ on $E$.
The conditions are satisfied, however, if one of the following additional conditions holds:
\begin{enumerate}
\item[\rm(a)] $E$ is a type 2 Banach space,
\item[\rm(b)] $S$ restricts to a $C_0$-semigroup on $H$,
\item[\rm(c)] $S$ is an analytic $C_0$-semigroup on $E$.
\end{enumerate}
We refer to \cite{DNW,NeeWei05a} for the easy proofs.

We are now in a position to state the main result of this section.
We use the notations introduced above, and let $\mu$ and $\mu^{(n)}$ denote the
Gaussian measures on $E$ associated
with the operators $R_{\Phi}$ and $R_{\Phi^{(n)}}$, respectively.

\begin{theorem}\label{thm:c-bis} 
Suppose that the equivalent conditions of Proposition \ref{prop:RPhi} are
satisfied. The following assertions are equivalent:
\ben
\item $\limn U_x^{(n)}(T) = U_x(T)$ in $L^p(\Om;E)$ for some (all) $x\in E$ and
some (all) $1\le p<\infty$;
\item $\limn R_{\Phi^{(n)}} = R_{\Phi}$ in $\g(L^2(0,T;H),E)$;
\item $\limn \mu^{(n)}=\mu$ weakly.
\een
In this situation we have $\limn U_x^{(n)}(t) = U_x(t)$ in $L^p(\Om;E)$ for all
$x\in E$, $t\in [0,T]$, and $1\le p<\infty$, and in fact we have
$$ \sup_{0\le t\le T}\E \n U_x^{(n)}(t) - U_x(t)\n^p \le 
 \sup_{0\le t\le T} \n S^{(n)}(t)x- S(t)x\n +
\E \n U^{(n)}(T) - U(T)\n^p,$$
where, as before, $U^{(n)} = U_0^{(n)}$ and $U=U_0$ correspond to the initial
value $0$. 
\end{theorem}

\begin{proof}
We begin by proving the equivalence of (1), (2), (3). Clearly it suffices to
consider the initial value $x=0$. 

For a given $1\le p<\infty$, a sequence of $E$-valued centred Gaussian random
variables converges in
$L^p(\Om;E)$ if and only if it converges in probability in $E$. Therefore, if (1)
holds for some $1\le p<\infty$, then it holds for all $1\le p<\infty$.

Taking $p=2$ in (1) the equivalence (1)$\Leftrightarrow$(2) follows from the
identity \eqref{eq:Ito} 
and the representations \eqref{eq:SI1} and \eqref{eq:SI2}.

Next we claim that $\limn R_{\Phi^{(n)}}\s x\s = R_\Phi\s x\s$ in $L^2(0,T;H)$
for all $x\s\in E\s$. Once we have shown this, the equivalence 
(2)$\Leftrightarrow$(3)
follows from 
\cite[Theorem 3.1]{HaaNee} (or by using the argument of \cite[page 18ff]{Nei}).
To prove the claim we fix $x\s\in E\s$ and note that in $L^2(0,T;H)$ we have
$$ R_\Phi\s x\s = i\s S\s(\cdot)x\s, \quad R_{\Phi^{(n)}}\s x\s =\sum_{j=1}^n
1_{I_j^{(n)}}(\cdot)\otimes i\s S\s(t_j^{(n)})x\s.$$
The inclusion mapping $i:H\embed E$ is $\g$-radonifying and hence compact.
As a consequence, the weak$\s$-continuity of $t\mapsto S\s(t)x\s$ implies 
that $t\mapsto i\s S\s(t)x\s = \Phi\s(t) x\s$ is continuous on $[0,T]$. 
It follows that $\limn R_{\Phi^{(n)}}\s(\cdot) x\s = R_\Phi\s(\cdot) x\s$
in $L^\infty(0,T;H)$, and hence in $L^2(0,T;H)$.

The final assertion an immediate consequence of \eqref{eq:SI1}, \eqref{eq:SI2},
and covariance domination \cite[Corollary 4.4]{NeeWei05a}.
\end{proof}

The assertions (1), (2), (3) are equivalent to the validity of a Lie--Trotter
product formula for the Ornstein-Uhlenbeck semigroup ${\mathscr P} =
\{\mathscr{P}(t)\}_{t\ge 0}$ associated with the problem \eqref{sACP},
which is defined on
the space $C_{\rm b}(E)$ of all bounded real-valued continuous functions on $E$
by the formula
$$ \mathscr{P}(t)f(x)= \E f(U_x(t)), \qquad x\in E, \ t\ge 0,$$
where $U_x$ is the solution of \eqref{sACP}.
In order to explain the precise result,  let us denote by
$\mathscr{S}=\{{\mathscr S}(t)\}_{t\ge 0}$ and $\mathscr{T}=\{{\mathscr
T}(t)\}_{t\ge 0}$ the semigroups on $C_{\rm b}(E)$ corresponding to the drift
term and 
the diffusion term in \eqref{sACP}. Thus,
$$
\begin{aligned}
{\mathscr S}(t)f(x) & = f(S(t)x), \\
{\mathscr T}(t)f(x) &  = \E f(x+W(t)),
\end{aligned}
\qquad t\ge 0, \ x\in E.
$$
Each of the semigroups $\mathscr{P}$, $\mathscr{S}$ and $\mathscr{T}$ is jointly
continuous in $t$ and $x$, uniformly on $[0,T]\times K$ for all compact sets
$K\subseteq E$. It was shown in \cite{KueNee} that if condition (3) of Theorem
\ref{thm:c-bis} holds, then for all $f\in C_{\rm b}(E)$ we have the Lie--Trotter
product formula
\beq\label{eq:LT}
{\mathscr P}(t)f(x)= \lim_{n\rightarrow\infty}
\big[{\mathscr T}({t}/{n}){\mathscr S}({t}/{n})\big]^{n} f(x)
\eeq
with convergence uniformly on $[0,T]\times K$ for all compact sets $K\subseteq
E$. Conversely it follows from the proof of this result that 
\eqref{eq:LT} with $x=0$ implies condition (3) of Theorem \ref{thm:c-bis}.
In the same paper it was shown that \eqref{eq:LT} holds if at least one of the
next two conditions is satisfied:
\ben
\item[(a)] $E$ is isomorphic to a Hilbert space;
\item[(b)] $S$ restricts to a $C_0$-semigroup on $H$.
\een
Thus, either of these conditions implies the convergence $\limn U_x^{(n)}(t) =
U_x(t)$ in $L^p(\Om;E)$ for all $x\in E$ and $t\in [0,T]$ of the splitting
scheme. The proofs in \cite{KueNee} are rather involved.
A simple proof for case (b) has been subsequently obtained by Johanna
Tikanm\"aki (personal communication). In Theorems \ref{thm:b} and \ref{thm:a}
below we shall give simple proofs for both cases (a) and (b), based on the
Proposition \ref{prop:g-cont} and an elementary convergence result for
$\g$-radonifying operators from \cite{NVW07a}, respectively. Moreover, case (a)
is extended to Banach spaces with type $2$.
Recall that a Banach space is said to have {\em type $1\le p\le 2$} if there
exists a constant $C\ge 0$ such that for all
finite choices $x_1,\dots, x_k\in E$ we have
$$ \Big(\E \Big\n \sum_{j=1}^k \g_j x_j \Big\n^2\Big)^\frac12 \le C
\Big(\sum_{j=1}^k \n x_j\n^p\Big)^\frac1p.$$
Hilbert spaces have type $2$ and $L^p$-spaces ($1\le p<\infty$) have type
$\min\{p,2\}$. We refer to \cite{AlbKal} for more information.

\begin{theorem}\label{thm:b}
If $E$ has type $2$, then the equivalent conditions of Proposition
\ref{prop:RPhi} and Theorem \ref{thm:c-bis} hold for every $C_0$-semigroup $S$
on $E$. As a consequence we have
$$\limn \sup_{t\in [0,T]}\E \n U^{(n)}(t) - U(t)\n^p = 0, \quad 1\le p<\infty.$$
\end{theorem}
\begin{proof}
By Proposition \ref{prop:g-cont} we have $\Phi\in C([0,T];\g(H,E))$. This
clearly implies that $\limn \Phi^{(n)} = \Phi$ in $L^\infty(0,T;\g(H,E))$, and 
hence in $L^2(0,T;\g(H,E))$. Since $E$ has type $2$, by \cite[Lemma
6.1]{NeeWei05b} the mapping $\Psi\mapsto R_\Psi$ defines a continuous inclusion 
$L^2(0,T;\g(H,E))\embed \g(L^2(0,T;H),E)$. It follows that 
$\limn R_{\Phi^{(n)}} = R_{\Phi}$ in $\g(L^2(0,T;H),E)$.
\end{proof}

\begin{theorem}\label{thm:a}
If $S$ restricts to a $C_0$-semigroup on $H$, then the equivalent conditions of
Proposition \ref{prop:RPhi} and Theorem \ref{thm:c-bis} hold. As a consequence
we have
$$\limn \sup_{t\in [0,T]} \E \n U^{(n)}(t) -U(t)\n^p =0, \quad 1\le p<\infty.$$
\end{theorem}
\begin{proof}
Let $S_H$ denote the restricted semigroup on $H$. 
From the identity $S(t)\circ i = i\circ S_H(t)$ we
have $R_{\Phi} = i\circ T$ and $R_{\Phi^{(n)}} = i\circ T^{(n)}$, where
$T$ and $T^{(n)}$ are the bounded operators from $L^2(0,T;H)$ to $H$ defined
by $$ T f := \int_0^T S_H(t)f(t)\,dt, \quad  
   T^{(n)} f = \int_0^T \sum_{j=1}^n 1_{I_j^{(n)}}(t) S_H(t_j^{(n)})f(t)\,dt.
$$
Since $\limn (T^{(n)})^* h = T^* h$ for all $h\in H$ by the strong continuity of
the adjoint semigroup $S_H\s$ (see \cite{Phi55}), it follows from
\cite[Proposition 2.4]{NVW07a}
that $\limn R_{\Phi^{(n)}} = R_{\Phi}$ in $\g(L^2(0,T;H),E)$.
\end{proof}

\section{Proof of Theorem \ref{thm2b}}
\label{sec:analytic}

In  this section we shall prove convergence of the splitting scheme under the
assumption that the $C_0$-semigroup generated by $A$ is analytic; no assumptions
on the space $E$ are made.
In this situation we are also able to give explicit rates of convergence in
suitable interpolation spaces.

We begin with a minor extension of a result due to Kalton
and Weis \cite{KalWei07}. It enables us to check whether certain $\calL(H,E)$--valued
functions define operators belonging to $\g(L^2(0,T;H),E)$. We refer to 
\cite[Section 13]{Nee-survey} for a detailed proof.

\begin{proposition}\label{prop:C1-derivative}
Let $\Phi: (a,b) \to \g(H,E)$ be continuously
differentiable with $$\int_a^b (s-a)^{\frac12}
\n{\Phi'(s)}\n_{\g(H,E)}\, ds < \infty.$$ 
Define $R_\Phi: L^2(a,b;H)\to E$ by
$$ R_\Phi f:= \int_a^b \Phi(t)f(t)\,dt.$$
Then $R_\Phi \in \gamma (L^2(a,b;H),E)$ and
\[ \n R_\Phi\n_{\gamma (L^2(a,b;H),E)} \leq
(b-a)^\frac12\n{\Phi(b)}\n_{\g(H,E)}+\int_a^b(s-a)^\frac12
 \n \Phi'(s)\n _{\g(H,E)}\, ds. 
\]
\end{proposition}

For $\a\geq 0$ and large enough $w\in \R$ we define $$E_\a:=
\mathscr{D}((w-A)^\a),$$ which is known to be independent of the choice of $w$.
It is a Banach space with respect to the norm $\n x\n_{E_\a}:= \n
(w-A)^{\a}x\n$. This norm depends of course on $w$, but any two such norms are
mutually equivalent. In what follows we consider $w$ to be fixed.\par
We shall also need the extrapolation spaces $E_{-\a}$, defined for $\a>0$ as the closure of $E$ with respect to the norm $\n
x\n_{E_{-\a}}:= \n (w-A)^{-\a}x\n$. It follows readily from the definitions that
for any two $\a,\b\in \R$ the operator $(w-A)^{\a}$ defines an isomorphism from
$E_{\b}$ onto $E_{\b-\a}$. 

In the next two remarks we fix $\a,\b\geq 0$ and $i\in \gamma(H,E_{\b})$, 
and suppose that $S$ is an analytic $C_0$-semigroup on $E$ with generator $A$.

\begin{remark}\label{r:PhiInGamma}
\rm By \cite[Theorem 2.6.13(c)]{Paz} one has, for any $\theta\geq 0$,
\begin{equation}\begin{aligned}\label{analytic}
\n S(t)\n_{\calL(E,E_{\theta})} \lesssim t^{-\theta}
\end{aligned}\end{equation}
with implied constant independent of
$t\in [0,T]$. 
From this and the ideal property for $\g$-radonifying operators we obtain the following 
estimate for $\Phi(t):= S(t)\circ i$:
$$\begin{aligned}
\n \Phi'(t) \n_{\g(H,E_{\a})} & \leq \n A S(t)\n_{\mathscr{L}(E_{\b},E_{\a})} \n i \n_{\g(H,E_{\b})}\\
& = \n S(t)\n_{\mathscr{L}(E,E_{\a+1-\b})} \n i \n_{\g(H,E_{\b})} \lesssim t^{-(\a+1-\b)^+} \n i \n_{\g(H,E_{\b})}
\end{aligned}
$$
where $r^+:=\max\{0,r\}$ for $r\in\R$; the implied constant is independent of 
$t\in [0,T]$ and $i\in \g(H,E_\b)$.
If $\a-\b<\frac12$, 
it then follows from Proposition \ref{prop:C1-derivative} that 
$$
\n R_\Phi\n_{\g(L^2(0,t;H),E_\a)} \lesssim t^{\min\{\frac12-\a+\b,\frac32\}} \n i\n_{\g(H,E_\b)}, 
$$ 
with implied constant independent of $t\in [0,T]$ and $i\in\g(H,E_\b)$. 
In particular, taking $\a=\b=0$ we see that the equivalent conditions 
of Proposition \ref{prop:RPhi} hold.
\end{remark}

\begin{remark}\label{r:PhiInGamma2}
\rm Suppose that $\d \in [0,\tinv{2})$. Identifying operator-valued functions with 
the integral operators they induce, we have
$$\begin{aligned}
&\n s\mapsto s^{-\d}S(t-s)i \n_{\gamma(L^2(0,t;H),E_{\a})}\\
&\qquad \leq \n s\mapsto s^{-\d}S(t-s)i \n_{\gamma(L^2(0,\frac{t}{2};H),E_{\a})} + \n s\mapsto (t-s)^{-\d}S(s)i \n_{\gamma(L^2(0,\frac{t}{2};H),E_{\a})}.
\end{aligned}$$
Applying Proposition \ref{prop:C1-derivative} to both terms on the right-hand side, if $\a-\b<\frac12$ it follows that $$[s\mapsto s^{-\d}S(t-s)i] \in \gamma(L^2(0,t;H),E_{\a})$$ for all $t\in [0,T]$.
\end{remark}

We need to introduce the following terminology.
Let $E$ and $F$ be Banach spaces. A family of operators $\mathscr{R}\subseteq
\calL(E,F)$ is called {\em $\g$-bounded} if there exists a finite constant $C\ge
0$ such that for all
finite choices $R_1,\dots,R_N \in  \mathscr{R}$ and vectors $x_1,\dots,x_N\in E$
we have
$$ \E \Big\n \sum_{n=1}^N \g_n R_n x_n\Big\n^2 \le C^2  
\E \Big\n \sum_{n=1}^N \g_n x_n\Big\n^2.$$
The least admissible constant $C$ is called the {\em $\g$-bound} of
$\mathscr{R}$, notation $\g(\mathscr{R})$. 
We refer to \cite{CPSW, DHP, KunWei, Wei} for examples and more
information. In these references the related notion of $R$-boundedness
is discussed; this notion is obtained by replacing the Gaussian random variables
by Rademacher variables in the above definition. Any $R$-bounded set is also
$\gamma$-bounded, and the two notions are equivalent if $E$ has finite cotype. \par
We continue with a multiplier result, also due to Kalton
and Weis \cite{KalWei07}. We refer to
\cite[Section 5]{Nee-survey} for a detailed proof.

\begin{proposition}\label{prop:KW}
Suppose that $E$ and $F$ are Banach spaces and $M:(0,T)\to \calL(E,F)$ is a
strongly measurable function
(in the sense that $t\mapsto M(t)x$ is strongly measurable for every $x\in E$)
with $\gamma$-bounded range $\mathscr{M}=\{M(t): \ t\in (0,T)\}$. 
Then for every finite rank simple function $\Phi:(0,T)\to \g(H,E)$ the operator 
$R_{M\Phi}$ belongs to $\g_\infty(L^2(0,T;H),F)$ and 
$$
 \n R_{M\Phi}\n_{\g_\infty(L^2(0,T;H),F)} \le \g(\mathscr{M})\,\n
R_\Phi\n_{\g(L^2(0,T;H),E)}.
$$
As a result, the map $\widetilde M: R_\Phi\mapsto R_{M\Phi}$ has a unique
extension to a bounded operator
$$ \widetilde M:  \g(L^2(0,T;H),E) \to \g_\infty(L^2(0,T;H),F)$$ of norm $\n \widetilde
M\n \le \g(\mathscr{M})$.
\end{proposition}

In the applications of this result below it will usually be possible to check that
actually we have $R_{M\Phi}\in \g(L^2(0,T;H),F)$. 

We will also need the following sufficient condition for $\g$-boundedness, which
is a variation of a result of Weis \cite[Proposition 2.5]{Wei}.

\begin{proposition}\label{prop:integr-der}
Let $E$ and $F$ be Banach spaces, and let $f:(0,T)\to \calL(E,F)$ 
be an function such that for all $x\in E$ the function $t\mapsto f(t)x$ is
continuously differentiable with integrable derivative.
Then the set $\mathscr{F}: = \{f(t): \ t\in (0,T)\}$ is
$\g$-bounded in $\mathscr{L}(E,F)$ and
$$ \g(\mathscr{F}) \le \n f(0+)\n + \n f'\n_1.$$
\end{proposition}

Here is a simple application:

\begin{lemma}\label{lem:analyticRbound} Let the $C_0$-semigroup $S$ be analytic
on $E$.

\begin{enumerate}\item[\rm(1)] For all $0\le \a<\d$ and $t\in (0,T]$ the set 
$\mathscr{S}_{\a,\d,t} = \{s^\d S(s): \ s\in [0,t]\}$ is $\g$-bounded
in $\calL(E,E_\a)$ and we have $$\gamma(\mathscr{S}_{\a,\d,t})\lesssim
t^{\d-\a}, \quad t\in (0,T],$$
with implied constant independent of $t\in (0,T]$. 
\item[\rm(2)] For all $0< \a \le 1$ 
the set 
$\mathscr{T}_{\a,t} = \{S(s)-I: \ s\in [0,t]\}$ is $\g$-bounded
in $\calL(E_\a,E)$ and we have $$\gamma(\mathscr{T}_{\a,t})\lesssim t^{\a},
\quad t\in [0,T],$$
with implied constant independent of $t\in [0,T]$.
\end{enumerate}\end{lemma}

\begin{proof}
For the proof of (1) we refer to \cite{DNW} or \cite[Lemma 10.17]{ISEM}. To
prove (2) it will be shown that for any fixed and large enough $w\in \R$ the
set 
$$\mathscr{T}_{\a,t}^w := \{e^{-ws}S(s)-I: \ s\in [0,t]\}$$ is $\g$-bounded in
$\calL(E_\a,E)$ with $\g$-bound $\lesssim t^\a$.
From this we deduce that $\{S(s): \ s\in [0,t]\}$ is $\g$-bounded in
$\calL(E_\a,E)$ with $\g$-bound $\lesssim 1$.  
In view of the identity $$S(s)-I = (e^{-w s}S(s) - I) + (1-e^{-ws})S(s)$$
and noting that $1-e^{-ws}\lesssim s$, this will prove the assertion of the
lemma.

For all $x\in E$ and $0\le s\le t$,
$$e^{-ws}S(s)x- x=\int_{0}^{s}e^{-wr}(A-w)S(r)x\, dr.$$
By \eqref{analytic} and Proposition \ref{prop:integr-der} the set
$\mathscr{T}_{\a,t}^w$  is $\g$-bounded in $\calL(E_\a,E)$ and 
$ \gamma(\mathscr{T}_{\a,t}^w)\lesssim \int_{0}^{t} s^{\a-1}\, ds \lesssim
{t^{\a}}.$
\end{proof}
 
We shall again write $U=U_0$ and $U^{(n)} = U_0^{(n)}$ for the solution of
\eqref{sACP0} and its approximations by the splitting scheme.
 
\begin{theorem}\label{thm:c} 
Assume that the semigroup $S$ is analytic on $E$ and that
$W$ is a Brownian motion in $E_\b$ for some $\b\ge 0$.
Then the equivalent conditions of Proposition \ref{prop:RPhi} and Theorem
\ref{thm:c-bis} hold. Moreover, for all $\a\ge 0$ and $0\le \theta\le 1$ such
that $\a-\b+\theta<\frac12$, and all $t\in [0,T]$ we have 
\begin{equation}\begin{aligned}\label{rconv}
\Vert R_{\Phi^{(n)}} - R_{\Phi} \Vert_{\g(L^2(0,t;H),E_\a)} \lesssim n^{-\theta}
t^{\frac12-(\a-\b+\theta)^+}.
\end{aligned}\end{equation}
with implied constant independent of $n\ge 1$ and $t\in [0,T]$.
As a consequence, for all $1\le p<\infty$ the solution $U$ of \eqref{sACP0}
satisfies
\beq\label{oconv} \big(\E \n U^{(n)}(t) - U(t)\n_{E_\a}^p\big)^\frac1p  \lesssim
n^{-\theta} t^{\frac12-(\a-\b+\theta)^+}
\eeq
with implied constant independent of $n\ge 1$ and $t\in [0,T]$.
\end{theorem}

\begin{proof}
The estimate \eqref{oconv} follows from \eqref{rconv}
via Theorem \ref{thm:c-bis}.

By rescaling time we may assume that $T=1$.
Let $\a,\b,\theta$ be as indicated. 
We begin by noting that the embedding $i:H\embed E$ associated with $W$ belongs
to $\gamma(H,E_{\b})$.

Pick $(\a-\b+\theta)^+<\d<\inv{2}$. Note that for $0<s\le T$ we have
$S^{(n)}(s)=S(n^{-1}\lceil ns \rceil)$ and $s\leq n^{-1}\lceil ns \rceil$, so
one can write, for all $n\ge 1$, 
\beq\label{eq:fact}
\Phi^{(n)}(s)-\Phi(s)= s^{\d}S(s)\circ\big(S(n^{-1}\lceil
ns\rceil-s)-I\big)\circ s^{-\d}i.
\eeq
Fix $t\in (0,1]$.
By the first part of  Lemma \ref{lem:analyticRbound} the set
$$\mathscr{S}_{\d} = \{s^\d S(s): \ s\in [0,t]\}$$ is $\g$-bounded
in $\calL(E,E_{(\alpha-\b+\theta)^+})$ (hence in
$\calL(E,E_{\alpha-\b+\theta})$,
hence in $\calL(E_{\b-\theta},E_{\alpha})$, with the same upper bounds for
the $\gamma$-bounds,  because $S(t)$ commutes with the fractional powers of $A$)
and we have
\begin{equation}\begin{aligned}\label{thm:c:h2}
\gamma(\mathscr{S}_{\d})\lesssim t^{\d-(\a-\b+\theta)^+}.
\end{aligned}\end{equation}
By the second part of Lemma \ref{lem:analyticRbound} the set
$$\mathscr{T}_{\theta,\frac1n} = \{S(s)-I: \ s\in [0,n^{-1}]\}$$ is $\g$-bounded
in $\calL(E_\theta,E)$ (and hence in $\calL(E_{\b},E_{\b-\theta})$, with the
same estimate for the $\g$-boundedness constant), and we have
\begin{equation}\begin{aligned}\label{thm:c:h3}
\gamma(\mathscr{T}_{\theta,\frac1n})\lesssim n^{-\theta}.
\end{aligned}\end{equation}
Using \eqref{eq:fact}, Remark \ref{r:PhiInGamma2}, Proposition \ref{prop:KW},  the identity
\begin{equation}\begin{aligned}\label{thm:c:h1}
\n R_{s\mapsto s^{-\d} i} \n_{\gamma(L^2(0,t;H),E_{\b})}&= 
\n s\mapsto s^{-\d}\n_{L^2(0,t)}
 \n i \n_{\gamma(H,E_{\b})} \eqsim t^{\inv{2}-\d} \n i \n_{\gamma(H,E_{\b})}
\end{aligned}\end{equation}
together with the estimates
\eqref{thm:c:h2}, and \eqref{thm:c:h3}, and noting that $n^{-1}\lceil ns \rceil
-s \leq n^{-1} $, one obtains
$$\begin{aligned}
\n R_{\Phi^{(n)}}-R_{\Phi} \n_{\gamma(L^2(0,t;H),E_{\a})}& \leq
\gamma(\mathscr{S}_{\d})\gamma(\mathscr{T}_{\theta,\frac1n}) \n R_{s\mapsto
s^{-\d} i} \n_{\gamma(L^2(0,t;H),E_{\b})}\\
& \lesssim n^{-\theta} t^{\frac12-(\a-\b+\theta)^+}\n i \n_{\gamma(H,E_{\b})}.
\end{aligned}$$
\end{proof}

\begin{remark}
\rm The condition $\a-\b+\theta<\frac12$ implies, in view of the restriction
$0\le \theta\le 1$, that $\a -\b<\frac12$.  
For $\a-\b<-\frac12$, Theorem \ref{thm:c} gives a rate of convergence of order
$n^{-1}$, whereas for $-\frac12\le \a-\b<\frac12$
 we obtain the rate $n^{-\theta}$ for any $0\le \theta<\frac12-\a+\b$. 
For $-\frac12< \a-\b<\frac12$ one can in fact obtain a slightly better rate at
the final time $T$, namely $(\ln\ln n)/n^{(\inv{2}-\a+\b)}$. More precisely, for
$n\ge 3$ we have
\begin{equation}\begin{aligned}\label{rconv-bis}
\Vert R_{\Phi^{(n)}} - R_{\Phi} \Vert_{\g(L^2(0,T;H),E_\a)} \lesssim
\frac{\ln\ln n}{n^{\inv{2}-\a+\b}} 
\end{aligned}\end{equation}
with constants independent of $n\ge 1$.\par
Once again observe that by scaling we may (and do) assume that $T=1$. In order
to prove \eqref{rconv-bis} we first give an estimate for a given time interval
$[a,b]$ where $0<a<b\leq 1$. In that case, for $\d> \a-\b+1$ one has
\beq\label{rconv1}
\Vert R_{\Phi^{(n)}} - R_{\Phi} \Vert_{\g(L^2(a,b;H),E_\a)} \lesssim n^{-1}
a^{\frac12-\d}b^{\d-\a-\b+1}
\eeq
with implied constant independent of $n\ge 1$ and $0<a<b\leq 1$. 
The proof of \eqref{rconv1} is similar to that of \eqref{rconv}, the main
difference being that we no longer need to ensure the square integrability of  
$s\mapsto s^{-\d}$ near $s=0$ in \eqref{thm:c:h1}. 
The details are as follows. Fix $n\ge 1$ and $0<a< b\le 1$ and pick an arbitrary
$\d>\a-\b+1$. Then, 
\beq\label{thm:c:h1b}
\begin{aligned}
\n R_{s\mapsto s^{-\d} i} \n_{\gamma(L^2(a,b;H),E_{\b})}
 = \n s\mapsto s^{-\d}\n_{L^2(a,b)} \n i \n_{\gamma(H,E_{\b})}
\lesssim a^{\inv{2}-\d} \n i \n_{\gamma(H,E_{\b})},
\end{aligned}
\eeq
with implied constant independent of $a\in (0,1]$ and $b\in (a,1]$;
the last inequality uses that $\d\ge \frac12$.
As in the proof of Theorem \ref{thm:c} the set $\mathscr{T}_{\frac1n} :=
\{S(s)-I: \ s\in [0,n^{-1}]\}$ is $\g$-bounded in $\calL(E_{\b},E_{\b-1})$, with
$\g$-bound
\begin{equation}\begin{aligned}\label{thm:c:h2b}
\gamma(\mathscr{T}_{\frac1n})\lesssim n^{-1}.
\end{aligned}\end{equation}
Finally, since $\d>\a-\b+1$, as in the proof of Theorem \ref{thm:c} 
the set 
$\mathscr{S}_{\d} = \{s^\d S(s): \ s\in [a,b]\}$ is $\g$-bounded
in $\calL(E_{\b-1},E_{\alpha})$ with
\begin{equation}\begin{aligned}\label{thm:c:h3b}
\gamma(\mathscr{S}_{\d})\lesssim b^{\d- \a-\b+1}.
\end{aligned}\end{equation}
Combining \eqref{thm:c:h1b}, \eqref{thm:c:h2b}, and \eqref{thm:c:h3b} we obtain
$$
\bal
\Vert R_{\Phi^{(n)}} - R_{\Phi} \Vert_{\g(L^2(a,b;H),E_\a)} 
 \lesssim n^{-1} a^{\frac12-\d}b^{\d-\a-\b+1} \n i \n_{\g(H,E_{\b})}.
\eal$$
Returning to the proof of estimate \eqref{rconv-bis} we fix an integer $n\ge
3$. 
Because $\b-\a < \inv{2}$ one can pick $\d>0$ such that $1+\a-\b < \d \leq
\frac{3}{2} +2(\a-\b)$. 
For $j=0,1,\dots$ define $a_j:=n^{-1+2^{-j}}$. Note that $a_0=1$ and $\limj
a_j=n^{-1}$.  
If in \eqref{rconv1} we take $a=a_j$ and $b=a_{j-1}$ we obtain the estimate
$$\begin{aligned}
\ & \n R_{\Phi^{(n)}}-R_{\Phi} \n_{\gamma(L^2(a_{j},a_{j-1};H),E_{\a})} \\ &
\qquad \lesssim
n^{-1+(-1+2^{-j})(\inv{2}-\d)+(-1+ 2^{1-j})(\d-\a+\b-1)} \n i\n_{\g(H,E_{\b})}
\\
& \qquad = n^{-\inv{2}+\a-\b+2^{-j}(\d -\frac{3}{2}-2(\a-\b))}\n
i\n_{\g(H,E_{\b})}\\ 
& \qquad \leq n^{-\inv{2}+\a-\b}\n i\n_{\g(H,E_{\b})},
\end{aligned}$$
where the last inequality used that $\d \leq \frac{3}{2} +2(\a-\b)$.
Set $k_n=\lceil (\ln\ln n)/\ln 2 \rceil$, so that $ a_{k_n}\leq en^{-1}$. Using
this estimate for $a_{k_n}$, from Theorem \ref{thm:c} we obtain, for any choice
of $0\le\theta<\inv{2}-\a+\b$ (which then satisfies $\theta<1$),
$$
\n R_{\Phi^{(n)}}-R_{\Phi} \n_{\gamma(L^2(0,a_{k_n};H),E_{\a})}
\lesssim
a_{k_n}^{\inv{2}-\a+\b-\theta} n^{-\theta}\n i\n_{\g(H,E_{\b})} 
\lesssim
n^{-\inv{2}+\a-\b}\n i\n_{\g(H,E_{\b})}.
$$
Combining the above one gets
$$
\bal
& \n R_{\Phi^{(n)}} - R_{\Phi} \n_{\g(L^2(0,1;H),E_\a)} \\ 
& \qquad \lesssim \n R_{\Phi^{(n)}} - R_{\Phi} \n_{\g(L^2(0,a_{k_n};H),E_\a)} 
 + \sum_{j=1}^{k_N} \n R_{\Phi^{(n)}} - R_{\Phi}
\n_{\g(L^2(a_{j},a_{j-1};H),E_\a)}\\
&\qquad \lesssim (1+\ln\ln n) n^{-\inv{2}+\alpha-\b} \n i\n_{\g(H,E_{\b})}.
\eal
$$
This gives the estimate \eqref{rconv-bis}.
\end{remark}

Under the assumptions that $S$ is analytic on $E$ and $W$ is a Brownian motion
on $E$, the solution $U$ of \eqref{sACP0} has a version with trajectories in
$C^\g([0,T];E_\a)$ for any $\a,\g\ge 0$ such that $\a+\g<\frac12$ \cite{DNW}.
The main result of this paper asserts that also the approximating processes
$U^{(n)}$ have trajectories in $C^\g([0,T];E_\a)$ and that the splitting scheme
converges with respect to the $C^\gamma([0,T];E_\a)$-norm, with a convergence
rate depending on $\a$ and $\g$ and the smoothness of the noise.

\begin{theorem}\label{thm:unif}
Let $S$ be analytic on $E$ and suppose that $W$ is a Brownian motion in $E_\b$
for some $\b\ge 0$. 
If $\a,\theta,\g\geq 0$ satisfy $\theta + \gamma < 1$ and $(\a-\b+\theta)^+ +\g
< \frac12$, 
then for all $1\leq p < \infty$ the solution $U$ of
\eqref{sACP0} satisfies
$$\begin{aligned}
\big(\E \n U^{(n)} - U \n_{C^{\gamma}([0,T],E_{\alpha})}^p\big)^{\inv{p}}
\lesssim n^{-\theta},
\end{aligned}$$
with implied constant independent of $n\ge 1$.
\end{theorem}

\begin{proof}
By scaling we may assume $T=1$. Put $V^{(n)}:=U^{(n)}-U$. Let $\a,\b,\g$ and
$\theta$ be as indicated. Without loss of generality we assume that $\g>0$. The
main step in the proof is the following claim.

\begin{claim}\label{claim}
There exists a constant $C$ such that for all $n\ge 1$, all $0\leq s < t \leq 1$
satisfying $t-s<\frac{1}{2n}$
 we have
$$
\big(\E \Vert V^{(n)}(t)-V^{(n)}(s) \Vert_{E_{\alpha}}^2\big)^\frac12 \le
Cn^{-\theta}(t-s)^{\g}.
$$
\end{claim}
\begin{proof}
Fix $n\ge 1$ and $0\leq s < t \leq 1$ such that $t-s<\frac{1}{2n}$.
Clearly,
\begin{equation}\label{unif:h1}
\begin{aligned}
\ \big(\E \Vert V^{(n)}(t)-V^{(n)}(s)
\Vert_{E_{\alpha}}^2\big)^\frac12 
 &  \leq \Big(\E \Big\n \int_{s}^{t} \Phi (t-r) - \Phi^{(n)}(t-r)
dW(r)\Big\n^{2}\Big)^{\frac12} \\
& \qquad + \Big(\E \Big\n \int_{0}^{s} \Phi (t-r) -  \Phi (s-r)
dW(r)\Big\n^{2}\Big)^{\frac12} \\
& \qquad + \Big(\E \Big\n \int_{0}^{s} \Phi^{(n)}(t-r) - \Phi^{(n)}(s-r)
dW(r)\Big\n^{2}\Big)^{\frac12}.
\end{aligned}
\end{equation}
For the first term we note that by \eqref{eq:Ito} (and the remark following it)
and \eqref{rconv} 
one has
\begin{equation}
\label{unif:h2}
\begin{aligned}
\ &\Big(\E \Big\Vert \int_{s}^{t} \Phi^{(n)} (t-r)- \Phi (t-r)\, dW_{H}(r)
\Big\Vert^2_{E_{\alpha}}\Big)^\frac12 \\
& \qquad = \Big(\E \Big\Vert \int_{0}^{t-s} \Phi^{(n)} (r)- \Phi (r)\, dW_{H}(r)
\Big\Vert^2_{E_{\alpha}}\Big)^\frac12 \\
&\qquad \lesssim n^{-\theta}(t-s)^{\inv{2}-(\a-\b+\theta)^{+}} \n
i\n_{\g(H,E_\b)} 
\\ & \qquad \le n^{-\theta} (t-s)^{\g} \n i\n_{\g(H,E_\b)}.
\end{aligned}
\end{equation}

The estimate for the second term is extracted from arguments in \cite{NVW08};
see also \cite[Theorem 10.19]{ISEM}. Fix $\eta>0$ such that $(\a-\b+\theta)^+
+\g < \eta < \frac12$. Then the set $\{t^\eta  S(t): \ t\in (0,T)\}$ is
$\g$-bounded in $\calL(E,E_{(\a-\b+\theta)^+ +\g})$ (hence in
$\calL(E,E_{\a-\b+\theta+\g})$, hence in 
$\calL(E_{\b-\theta -\g},E_\a)$) by the first part of Lemma
\ref{lem:analyticRbound}, and therefore 
\begin{equation}\label{unif:h3}
\begin{aligned}
\ &  \Big(\E \Big\n \int_0^s \Phi(t-r)-\Phi(s-r) \,dW_H(r)
\Big\n_{E_\a}^2\Big)^\frac12 
 \\ & \qquad = \Big(\E \Big\n \int_0^s [(s-r)^{\eta}S(s-r)]\circ 
[(s-r)^{-\eta} (S(t-s)-I)\circ i]\,dW_H(r) \Big\n_{E_{\a}}^2\Big)^\frac12 
 \\ & \qquad \lesssim \Big(\E \Big\n \int_0^s (s-r)^{-\eta} (S(t-s)-I)\circ
i)\,dW_H(r) \Big\n_{E_{\b-\theta-\g}}^2\Big)^\frac12 
 \\ & \qquad = \Big(\int_0^s (s-r)^{-2\eta}\,dr\Big)^\frac12\n (S(t-s)-I)\circ
i\n_{\g(H,E_{\b-\theta-\g})} 
 \\ & \qquad \lesssim \n S(t-s)-I\n_{\calL(E_\b, E_{\b-\theta-\g})} \n
i\n_{\g(H,E_\b)}
 \\ & \qquad \eqsim \n S(t-s)-I\n_{\calL(E_{\g+\theta},E)} \n i\n_{\g(H,E_\b)}
 \\ & \qquad \lesssim (t-s)^{\g+\theta}\n i\n_{\g(H,E_\b)}
 \\ & \qquad \lesssim n^{-\theta}(t-s)^{\g}\n i\n_{\g(H,E_\b)}. 
\end{aligned}
\end{equation}

To estimate the third term on the right-hand side of \eqref{unif:h1}, we first
define sets $B_0$ and $B_1$ by
$$\begin{aligned}
B_0:= &\,\{ r\in(0,s): \ S^{(n)}(t-r) = S^{(n)}(s-r)
\}
\\  = &\,\{ r\in(0,s): \ \lceil n(t-r)\rceil = \lceil n(s-r)\rceil\}
,\\
B_1:= &\,\{ r\in(0,s): \ S^{(n)}(t-r) = S(n^{-1})S^{(n)}(s-r)
\}
\\ = &\,\{ r\in(0,s): \ \lceil n(t-r)\rceil = \lceil n(s-r)\rceil +1\}.
\end{aligned}$$
Both equalities follow from the identity $S^{(n)}(u) = S(n^{-1}\lceil nu\rceil)$
for $u\in (0,T)$. 
By definition of $B_0$ and $B_1$ one has
\beq\label{unif:h3a}
\begin{aligned}
&\Big(\E \Big\Vert \int_{0}^{s} \Phi^{(n)} (t-r) - \Phi^{(n)} (s-r)\, dW_{H}(r)
\Big\Vert^2_{E_{\alpha}}\Big)^{\inv{2}} \\
&\qquad \leq \Big(\E \Big\Vert \int_{B_0}  \Phi^{(n)}(t-r) - \Phi^{(n)}(s-r) \,
dW_{H}(r) \Big\Vert^2_{E_{\alpha}}\Big)^{\inv{2}}\\
&\qquad\qquad + \Big(\E \Big\Vert \int_{B_1}  \Phi^{(n)}(t-r) - \Phi^{(n)}(s-r)
\, dW_{H}(r) \Big\Vert^2_{E_{\alpha}}\Big)^{\inv{2}}\\
&\qquad = \Big(\E \Big\Vert \int_{B_1}  S^{(n)}(s-r) ( S(n^{-1})- I ) i \,
dW_{H}(r) \Big\Vert^2_{E_{\alpha}}\Big)^{\inv{2}},
\end{aligned}
\eeq
noting that the integrand of the integral over $B_0$ vanishes.

Set $\d:=\theta+\g$. To estimate the right-hand side, observe that from
$\alpha-\b + \d < \frac12$ we may pick $\eta>0$ such that $\alpha-\b+\d < \eta<
\frac12$. 
Using the identity $S^{(n)}(u)=S(n^{-1}\lceil nu \rceil)$ and applying
Proposition \ref{prop:KW} and part (1) of Lemma \ref{lem:analyticRbound}, and
then using the estimate 
 $\n S(u) - I \n_{\calL(E_\d,E)}\lesssim u^{\d}$ and Proposition \ref{prop:gR},
we obtain
\beq\label{eq:thirdterm}
\begin{aligned}
\ & \Big(\E \Big\Vert \int_{B_1}  S^{(n)}(s-r) ( S(n^{-1})- I ) i \, dW_{H}(r)
\Big\Vert^2_{E_{\alpha}}\Big)^{\inv{2}}
\\ & \quad \eqsim \Big(\E \Big\Vert \int_{B_1}  (n^{-1}\lceil
n(s-r)\rceil)^{\eta} 
S(n^{-1}\lceil n(s-r) \rceil)
\\
& \quad\quad \times (n^{-1}\lceil n(s-r)\rceil)^{-\eta} 
( S(n^{-1})- I)i\, dW_{H}(r) \Big\Vert^2_{E_\a}\Big)^{\inv{2}}\\
& \quad \lesssim \Big(\E \Big\Vert \int_{B_1} (n^{-1}\lceil
n(s-r)\rceil)^{-\eta} 
( S(n^{-1})- I)i\, dW_{H}(r) \Big\Vert^2_{E_{\b-\d}}\Big)^{\inv{2}}\\
& \quad \lesssim 
n^{-\d} \Vert (s-\cdot)^{-\eta} \Vert_{L^{2}(B_1)} \Vert i
\Vert_{\gamma(H,E_\b)}.
\end{aligned}
\eeq
In order to estimate the 
$L^2(B_1)$-norm of the function $f_s(r) := (s-r)^{-\eta}$ we note that $B_1
\subseteq\bigcup_{j=1}^{n} B_1^{(j)}$,
where 
$$
\bal B_1^{(j)} 
&= \{r\in (0,s): \ s-r \le jn^{-1} < t-r\} 
\\ & = \{r\in (0,s): \ jn^{-1}-t+s < s-r \le jn^{-1} \} .
\eal$$
From this it is easy to see that $|B_1^{(j)} | \le t-s$ and that for $r\in
B_1^{(j)}$ one has $$(s-r)^{-2\eta}\leq (jn^{-1}-t+s)^{-2\eta} \le
n^{2\eta}(j-\tfrac12)^{-2\eta}$$ (the latter inequality following from
$t-s<1/2n$), and therefore 
$$\n f_s\n_{L^2(B_1)}^2 = \int_{B_1} |f_s(r)|^2\,dr  \le
n^{2\eta}|B_1^{(j)}|\sum_{j=1}^n \frac{1}{(j-\tfrac12)^{2\eta}}  
 \lesssim  n(t-s).
$$
As a consequence,
\beq\label{eq:B1} 
\n f_s\n_{L^2(B_1)}\lesssim n^\frac12 (t-s)^\frac12 = n^\frac12
(t-s)^{\frac12-\g}(t-s)^\g \lesssim n^\g (t-s)^\g.
\eeq
Combining the estimates \eqref{eq:thirdterm} and \eqref{eq:B1} and estimating
the non-negative powers of $s$ by $1$ we find
\beq\label{eq:est2}
 \Big(\E \Big\Vert \int_{B_1}  S^{(n)}(s-r) ( S^{(n)}(t-s)- I )i\, dW_{H}(r)
\Big\Vert^2_{E_{\alpha}}\Big)^{\inv{2}} \\
\lesssim n^{-\theta}(t-s)^{\g} \n i\n_{\g(H,E_\b)}.
\eeq
Claim \ref{claim} now follows by combining \eqref{unif:h1}, \eqref{unif:h2},
\eqref{unif:h3}, \eqref{unif:h3a} and \eqref{eq:est2}.
\end{proof}

We are now ready to finish the proof of the theorem. 
By the triangle inequality and Theorem \ref{thm:c}, for all $0\le s,t\le 1$ we
have
$$\begin{aligned}
\big(\E \Vert V^{(n)}(t)-V^{(n)}(s) \Vert_{E_{\alpha}}^2\big)^\frac12 
& \le 
\big(\E \Vert U^{(n)}(t)-U(t) \Vert_{E_{\alpha}}^2\big)^\frac12 +
\big(\E \Vert U^{(n)}(s)-U(s) \Vert_{E_{\alpha}}^2\big)^\frac12 
\\ &  \lesssim 
n^{-\delta}\n i\n_{\g(H,E_\b)}.
\end{aligned}$$
Hence if $t-s\geq (2n)^{-1}$ one has
\beq\label{eq:eg2}
\big(\E \Vert V^{(n)}(t)-V^{(n)}(s) \Vert_{E_{\alpha}}^2\big)^\frac12 \lesssim
n^{-\delta}\n i\n_{\g(H,E_\b)}\lesssim n^{-\theta}(t-s)^{\gamma}\n
i\n_{\g(H,E_\b)}.
\eeq
The random variables $V^{(n)}(t)$ being Gaussian, from the claim and
\eqref{eq:eg2} combined with the Kahane-Khintchine inequalities we deduce
that for all $1 \leq q < \infty$ and $0\le s<t\le 1$ one has
\begin{equation}\begin{aligned}\label{eq:concl}
\big(\E \Vert V^{(n)}(t) - V^{(n)}(s) \Vert^q_{E_{\a}}\big)^\frac1q \lesssim
n^{-\theta }(t-s)^{\gamma}\n i\n_{\g(H,E_\b)}. 
\end{aligned}\end{equation}

Now fix any $0<\g'<\g$ and take $1/\g'<q<\infty$. Then
by \eqref{eq:concl} and the Kolmogorov-Chentsov criterion with $L^q$-moments
(see \cite[Theorem 5]{FeyLaP99}),
$$
\n U^{(n)}-U\n_{L^q(\Om;C^{\g'-\frac1q}([0,T];E_\a))}
\lesssim \n U^{(n)}-U\n_{C^\g([0,T];L^q(\Om;E_\a))} 
\lesssim n^{-\theta}\n i\n_{\g(H,E_\b)}.
$$
This inequality shows that for all $0<\bar \g<\g$  we have
$$ \n U^{(n)}-U\n_{L^{q}(\Om;C^{\bar\g}([0,T];E_\a))} \lesssim n^{-\theta}\n
i\n_{\g(H,E_\b)}
$$
for all sufficiently large $1\le q<\infty$.
It is clear that once we know this, this inequality extends to all values $1\le
q<\infty$. 
This completes the proof of the theorem (with $\bar \g$ instead of $\g$, which
obviously suffices).\label{gammaarg}
\end{proof}
 
\begin{corollary}\label{cor:as} Suppose that $S$ is analytic on $E$ and
that $W$ is a Brownian motion in $E_\b$ for some $\b\ge 0$. Let $\a,\g,\theta
\ge 0$ satisfy $\theta+ \gamma < 1$ and $(\a-\b+\theta)^+ +\g<\frac12$.
Then for almost all $\om\in\Om$ there exists a constant $C(\om)$ such that
the solution $U$ of \eqref{sACP0} satisfies
$$\Vert U_x^{(n)}(\cdot,\om) - U_x(\cdot,\om) \Vert_{C^\g([0,T];E_{\alpha})} 
\le \frac{C(\omega)}{n^\theta} \ \hbox{ for all } \ n=1,2,\dots$$
\end{corollary}
\begin{proof} 
Set $$\Omega_n := \Big\{\om\in\Om: \ \Vert U^{(n)}(\cdot,\om) - U(\cdot,\om)
\Vert_{C^\g([0,T];E_{\alpha})} > \frac1{n^\theta}\Big\}.$$ 
Pick $\bar{\theta} >\theta$ in such a way that $0\le
\a-\b+\g+\bar\theta<\frac12$ and let $p\ge 1$ be so large that
$(\bar\theta-\theta)p>1$. 
By Theorem \ref{thm:unif}, applied  with $\bar\theta$ instead of $\theta$, and
Chebyshev's inequality, 
$$\P(\Om_n) \le n^{\theta p} \E \Vert U^{(n)}(\cdot,\om) - U(\cdot,\om)
\Vert_{C^\g([0,T];E_{\alpha})}^p \le \frac{C^p}{n^{(\bar\theta-\theta)p}}$$
with constant $C$ independent of $n$. By the choice of $p$ we have $\sum_{n\ge
1} \P(\Om_n)<\infty$, and therefore by the Borel-Cantelli lemma
$$\P(\{\om\in\Om: \ \om\in\Om_n\hbox{ infinitely often}\})=0.$$ For the $\om\in\Om$
belonging to this set we have
$$ C(\om):= \sup_{n\ge 1} n^\theta\Vert U^{(n)}(\cdot,\om) - U(\cdot,\om)
\Vert_{C^\g([0,T];E_{\alpha})} <\infty.$$
\end{proof}

We conclude this section with an application of our results to the stochastic
heat equation on the unit interval driven by space-time white noise.
This example is merely included as a demonstration how such equations can be
handled in the present framework. We don't strive for the greatest possible
generality. For instance, as in \cite{Brz97, DNW} the Laplace operator can be
replaced by more general second order elliptic operators. 

\begin{example}\label{ex:stoch-heat}
\rm Consider the following stochastic partial differential 
equation driven by space-time white noise $w$:
\begin{equation}\label{stochheat}
\left\{
\begin{aligned}
\frac{\partial{u}}{\partial t}(t,x) & 
= \Delta u(t,x) +
\frac{\partial w}{\partial t}(t,x),\quad  & x\in [0,1], \ t\in [0,T], \\
u(0,x) & = 0, \quad & x\in [0,1],\\
u(t,0) & =u(t,1)= 0, \quad & t\in [0,T].
\end{aligned}
\right.\end{equation}
Following the approach of \cite{DNW} we put $H:=L^2(0,1)$ and $F:=L^q(0,1)$,
where the exponent
$q\ge 2$ is to be chosen later on. 
In order to formulate the problem \eqref{stochheat} as an abstract stochastic
evolution equation of the form
\beq\label{SDE}
\left\{\begin{aligned}
dU(t) & = A U(t)\,dt + dW(t), \quad t\in [0,T], \\
 U(0) & = 0,
\end{aligned}
\right.
\eeq
where $W$ is a Brownian motion with values in a suitable Banach space $E$, we
fix an arbitrary real number $\rho<-\frac14$, to be chosen in a moment, 
and let $E:= F_{\rho}$ denote the extrapolation space 
of order $-\rho$ associated with the Dirichlet Laplacian 
in $F$. It is shown in \cite{DNW} (see also \cite[Lemma 6.5]{Brz97}) that 
the identity operator on $H$ extends to a $\g$-radonifying 
embedding from $H$ into $E$. As a result, the $H$-cylindrical Brownian motion
$W_H$ canonically associated with $w$ (see \eqref{eq:cylBM}) may be
identified with a Brownian motion $W$ in $E$. Furthermore
the extrapolated Dirichlet Laplacian, henceforth denoted by $A$,
satisfies the assumptions of Theorem \ref{thm:unif} in $E$.

Let $U$ be the solution of \eqref{SDE} in $E$. By definition, we shall regard
$U$ as the solution of \eqref{stochheat}. Suppose now that we are given  real
numbers $\g,\d,\theta\ge 0$
satisfy $$\g+\d+\theta<\tfrac14.$$
This ensures that one can choose $\a\ge 0$ and $\rho<-\frac14$ in such a way
that  
$\a+\rho>\d$ and $\a+\g+\theta < \frac12$.
By Theorem \ref{thm:unif} (with $\b=0$), for all $1\leq p < \infty$ the
splitting scheme associated with problem \eqref{SDE} satisfies 
$$\begin{aligned}
\big(\E \n U^{(n)} - U \n_{C^{\gamma}([0,T],E_{\alpha})}^p\big)^{\inv{p}}
\lesssim n^{-\theta}.
\end{aligned}$$
Putting $\eta:=\alpha+\rho$ we have
$E_\alpha= (F_\rho)_{\a} =F_{\eta}$, and this space embeds into $F$ since
$\eta>\d\geq 0$.

Choose $q\ge 2$ so large that $2\delta+\frac1q<2\eta$. We have
$$ F_{\eta} = H_0^{2\eta,q}(0,1) = \{f\in H^{2\eta,q}(0,1): \ f(0)=f(1)=0\}$$ 
with equivalent norms. By the Sobolev embedding theorem, 
$$H^{2\eta,q}(0,1)\embed C^{2\delta}[0,1]$$ 
with continuous inclusion.  Here
$C^{2\delta}[0,1]$ is the space of all H\"older continuous functions 
$f:[0,1]\to\R$ of exponent $2\delta$.
We denote $C_0^{2\delta}[0,1] = \{f\in C^{2\delta}[0,1]: f(0)=f(1)=0\}$. 
Putting things together we obtain a continuous inclusion
$$ F_{\eta}\embed C_0^{2\delta}[0,1].$$
\end{example}

We have proved the following theorem (cf. Example \ref{ex:stoch-heat-intro}).

\begin{theorem}\label{thm:stoch-heat} 
For all $0\le \d<\frac14$ the stochastic heat equation \eqref{stochheat}
admits a solution $U$ in $C_0^{2\delta}[0,1]$, and for all $\g,\theta\ge 0$
satisfying  
$\g+\delta+\theta<\frac14$ we have
$$\begin{aligned}
\big(\E \n U^{(n)} - U \n_{C^{\gamma}([0,T],C_0^{2\delta}[0,1])}^p\big)^{\inv{p}}
\lesssim n^{-\theta}.
\end{aligned}$$
\end{theorem}

By Corollary \ref{cor:as}, we also obtain almost sure convergence with respect
to
the norm of $C^{\gamma}([0,T],C_0^{2\delta}[0,1])$ with rate $1/n^\theta$.

\section{A counterexample for convergence}\label{sec:example}

We shall now present an example of a $C_0$-semigroup $S$ on a Banach space $E$
and an $E$-valued Brownian motion $W$ such that the problem \eqref{sACP0} admits
a solution with continuous trajectories whilst the 
associated splitting scheme fails to converge. Although the actual construction
is somewhat involved, the semigroup in this example is a translation 
semigroup on a suitable vector-valued Lebesgue space. Such semigroups occur naturally
in the context of stochastic delay equations.

We take $E = L^q(0,1;\ell^p)$, with $1\le p<2$ and $q\ge 2$, and consider the
$E$-valued Brownian motion $W_f = w\otimes f$, where $w$ is a standard
real-valued Brownian motion and $f\in E$ is a fixed element.
With this notation a function $\Psi: (0,1)\to\calL(E)$ is stochastically
integrable with respect to $W_f$ if and only if $\Psi f: (0,1)\to E$ is  
is stochastically integrable with respect to $w$, in which case we have
$$ \int_0^1 \Psi\,dW_f = \int_0^1 \Psi f\,dw.$$

Let $1\leq p < 2$ and $u>\frac2p$ be fixed. For $k=1,2,\ldots$ and $j=0,\ldots
2^{k-1}-1$ define the intervals $I_{k,j} =
(\frac{2j+1}{2^k},\frac{2j+1}{2^k}+2^{-uk}]$. As in particular $u>1$, for all
$k=1,2,\dots$ the intervals $I_{k,i}$ and $I_{k,j}$ are disjoint for $i\neq j$.
Let $0< r < 1- \frac{p}{2}$ and denote the basic sequence of unit vectors in
$\ell^p$ by $\{e_n\}_{n\ge 1}$. Inspired by \cite[Example 3.2]{RosSuc} we define
$f\in L^{\infty}(\mathbb{R};\ell^p)$ by
$$\begin{aligned}
f(t):=\sum_{k=1}^{\infty} \sum_{j=0}^{2^{k-1}-1}  2^{-\frac{r}{p}k}
1_{I_{k,j}}(t) e_{2^{k-1}+j}.
\end{aligned}$$
Observe that $f(t)=0$ for $t\in \mathbb{R}\setminus (0,1)$ and $f$ is
well-defined: because $I_{k,j}$ and $I_{k,i}$ are disjoint for $i\neq j$ one
has, for any $t\in (0,1)$,
$$\begin{aligned}
\Vert f(t) \Vert^p_{\ell^p} \leq \sum_{k=1}^{\infty}  2^{-rk} < \infty.
\end{aligned}$$
For a given interval $I=(a,b]$, $0\leq a <b<\infty$, we write $\Delta w_I :=
w(b) -w(a)$. 

\begin{claim}\label{c:fstochint}
The function $f$ is stochastically integrable on $(0,1)$ and
\beq\label{repstochintf}
\begin{aligned}
\int_{0}^{1} f(t) \,dw(t) &= \sum_{n=1}^{\infty} \int_{0}^{1}
\lb{f(t)},{e^*_n}\rb e_n \,dw(t)  \\
&= \sum_{k=1}^{\infty} \sum_{j=0}^{2^{k-1}-1} 2^{-\frac{r}{p} k} \Delta
w_{I_{k,j}} e_{2^{k-1}+j},
\end{aligned}
\eeq
where $\{e_n^*\}_{n\ge 1}$ is the basic sequence of unit vectors in $\ell^{p'}$,
$\inv{p}+\inv{p'}=1$.
\end{claim}
\begin{proof}
We shall deduce this from \cite[Theorem 2.3, $(3)\Rightarrow (1)$]{NeeWei05a}.
Define the $\ell^p$-valued Gaussian random variable $$X:= \sum_{k=1}^{\infty}
\sum_{j=0}^{2^{k-1}-1} 2^{-\frac{r}{p} k} \Delta w_{I_{k,j}} e_{2^{k-1}+j}.$$
This sum converges absolutely in $L^p(\Om;\ell^p)$.
Indeed, let $\gamma$ denote a standard Gaussian random variable. Then by
Fubini's theorem one has
$$\begin{aligned}
\E \Big\Vert \sum_{k=1}^{\infty} \sum_{j=0}^{2^{k-1}-1} 2^{-\frac{r}{p} k}
\Delta w_{I_{k,j}} e_{2^{k-1}+j} \Big\Vert^p_{\ell^p} &= \sum_{k=1}^{\infty}
\sum_{j=0}^{2^{k-1}-1} 2^{-rk} 2^{-\frac{u}{2}pk}\E|\gamma|^p \\
&= \sum_{k=1}^{\infty} 2^{k(1-r-\frac{u}{2}p)-1}\E|\gamma|^p < \infty.
\end{aligned}
$$
By the Kahane-Khintchine inequalities, the sum defining $X$ converges absolutely
in $L^q(\Omega;\ell^p)$ for all $1\le q<\infty$. 

For any linear combination $a^*=\sum_{n=1}^{N}a_n e^*_n\in \ell^{p'}$ one easily
checks that
$$\begin{aligned}
\lb{X},{a^*}\rb= \int_{0}^{1} \lb{f(t)},{a^*}\rb\, dw(t).
\end{aligned}$$
Hence by \cite[Theorem 2.3]{NeeWei05a}, $f$ is stochastically integrable and
\eqref{repstochintf} holds.
\end{proof}

By similar reasoning (or an application of \cite[Corollary 2.7]{NeeWei05a}), for
all $s\in\mathbb{R}$ the function $t\mapsto f(t+s)$ is stochastically integrable
on $(0,1)$ and
$$
\int_{0}^{1} f(t+s) \,dw(t) = \sum_{n=1}^{\infty} \int_{0}^{1}
\lb{f(t+s)},e_n\s\rb e_n\,dw(t). 
$$
Let $q\geq1$ and let $\{S(t)\}_{t\in\mathbb{R}}$ be the left-shift group on
$L^{q}(\mathbb{R};\ell^p)$ defined by
$$\begin{aligned}
(S(t)g)(s)&=g(t+s), \quad s,t\in\R,\ g\in L^{q}(\R;\ell^p).
\end{aligned}$$
\begin{claim}\label{c:Sfstochint}
For any $q\geq 1$ the $L^q(\mathbb{R};\ell^p)$-valued function $t\mapsto S(t)f$
is stochastically integrable on $(0,1)$ and
$$\begin{aligned}
\Big(\int_{0}^{1} S(t) f \,dw(t)\Big)(s) &= \int_{0}^{1} f(t+s) \,dw(t)
\end{aligned}$$
for almost all $s\in \mathbb{R}$ almost surely.
\end{claim}
\begin{proof}
For $s\not\in (-1,1)$ the function $t\mapsto f(t+s)$ is identically $0$ on 
$(0,1)$, and for $s\in (-1,1)$ we have
$$ \E \Big\n \int_0^1 f(t+s)\,dw(t)\Big\n_{\ell^p}^q \le  \E \Big\n \int_0^1
f(t)\,dw(t)\Big\n_{\ell^p}^q.$$ 
As a consequence, $L^q(\Omega;\ell^p)$-valued function
$s\mapsto \int_0^1 f(s+t)\,dw(t)$ defines an element of
$L^q(\R;L^q(\Omega;\ell^p))$. Under the natural isometry
$ L^q(\R;L^q(\Omega;\ell^p))\simeq L^q(\Omega;L^q(\R;\ell^p))$
we may identify this function with an element $Y\in L^q(\Omega;L^q(\R;\ell^p))$.
To establish the claim,  with an appeal to \cite[Theorem 2.3]{NeeWei05a}
it suffices to check that for all $a^*\in \ell^{p'}$ and Borel sets $A\in
\mathscr{B}(\mathbb{R})$ we have
$$
\int_{0}^{1} \lb{S(t)f},1_A\otimes {a^*}\rb \,dw(t) = \lb Y, 1_A\otimes a\s\rb. 
\label{i:int} 
$$
By writing out both sides, this identity is seen to be an immediate consequence
of the stochastic Fubini theorem
(see, e.g., \cite[Theorem 3.3]{NeeWei05a}).
\end{proof}

In the same way one sees that for $t\geq 0$ the stochastic integrals
$\int_{0}^{t} S(-s)f\,dw(s)$ are well-defined. Because the process $t\mapsto
\int_{0}^{t} S(-s)f\,dw(s)$ is a martingale having a continuous version by
Doob's maximal inequality, we also know that the convolution process
$$U(t):=\int_{0}^{t} S(t-s)f \,dw(s)=S(t)\int_{0}^{t} S(-s)f \,dw(s)$$ has a
continuous version. However, as we shall see, the splitting scheme for $U$ fails
to converge.\par
For $n\ge 1$ define
$$\begin{aligned}
S^{(n)}&:= \sum_{k=1}^{n} 1_{(\frac{k-1}{n},\frac{k}{n}]}
\otimes S\Big(\frac{k}{n}\Big).
\end{aligned}$$
Observe that for any $s,t\in\mathbb{R}$
\begin{equation}\begin{aligned}\label{Snf}
(S^{(n)}(t)f)(s)&= \sum_{k=1}^{n} 1_{(\frac{k-1}{n},\frac{k}{n}]}(t)f\Big(
\frac{k}{n} + s \Big).
\end{aligned}\end{equation}
Similarly to the above one checks that
$$\begin{aligned}
\Big( \int_{0}^{1} S^{(n)}(t)f \,dw(t)\Big)(s) & =
\sum_{k=1}^{n}f\Big(\frac{k}{n}+s\Big)
\Big[w\Big(\frac{k}{n}\Big)-w\Big(\frac{k-1}{n}\Big)\Big]
\end{aligned}$$
for almost all $s\in\R$ almost surely.

The clue to this example is that for $n$ fixed and $s\in (0,2^{-un}]$ the
function $t\mapsto (S^{(2^n)}(t)f)(s)$ always `picks up' the values of $f$ at
the left parts of the dyadic intervals where $f$ is defined to be non-zero. Thus
for these values of $s$ the function $t\mapsto (S^{(2^n)}(t)f)(s)$ it is nowhere
zero and its stochastic integral blows up as $n\rightarrow \infty$. We shall
make this precise. Our aim is to prove that for certain values of $q>2$ (to be
determined later on) one has
\begin{equation}\begin{aligned}\label{explode}
\E \Big\Vert \int_{0}^{1} S^{(2^n)} (t) f \,dw(t)
\Big\Vert^p_{L^q(\mathbb{R};\ell^p)} \rightarrow \infty \quad \textrm{as
}n\rightarrow \infty.
\end{aligned}\end{equation}
\par
By Minkowski's inequality we have, for any $n\ge 1$ and $q\geq p$,
$$\begin{aligned}
&\Big[\E \Big\Vert \int_{0}^{1} S^{(2^n)} (t) f \,dw(t)
\Big\Vert^p_{L^q(\mathbb{R};\ell^p)}\Big]^{\inv{p}} \\
& \qquad \geq \Big[\int_\mathbb{R}\Big( \E \Big\Vert \Big(\int_{0}^{1}
S^{(2^n)}(t)f\, dw(t)\Big)(s)\Big\Vert_{\ell^p}^p \Big)^{\frac{q}{p}}
\,ds\Big]^{\inv{q}}.
\end{aligned}$$
Now fix $n\ge 1$. For any $1\leq k \leq n$ and any $j=0,\ldots,2^{k-1}-1$ there
exists a unique $1\leq i_{k,j}\leq 2^n-1$ such that
$\frac{i_{k,j}}{2^n}=\frac{2j+1}{2^k}$. Now observe that by definition of $f$
one has for $s\in(0,2^{-un}]$ that
$$\begin{aligned}
\Big\lb
f\big(\frac{i_{k,j}}{2^n}+s\big),e^*_{2^{k-1}+j}\Big\rb=2^{-k\frac{r}{p}}.
\end{aligned}$$
Using this and representation \eqref{Snf} one obtains that for
$s\in(0,2^{-un}]$, $1\leq k \leq n$, $j=0,\ldots,2^{k-1}-1$, and any $t\in
\big(\frac{i_{k,j}-1}{2^n},\frac{i_{k,j}}{2^{n}}\big]=:I_{k,j}^{n}$,
\begin{equation}\begin{aligned}\label{Snfl}
\big\lb{(S^{(2^n)}(t)f)(s)},e^*_{2^{k-1}+j}\rb=2^{-k\frac{r}{p}}.
\end{aligned}\end{equation} To prove
\eqref{explode}, we now estimate 
$$\begin{aligned}
&\int_{\mathbb{R}}\Big( \E \Big\Vert \Big(\int_{0}^{1} S^{(2^n)}(t)f\,
dw(t)\Big)(s)\Big\Vert_{\ell^p}^p \Big)^{\frac{q}{p}} \,ds\\
& \qquad \geq \int_{0}^{2^{-un}}\Big( \sum_{k=1}^{n}\sum_{j=0}^{2^{k-1}-1} \E
\Big| \int_{0}^{1} \big\lb{(S^{(2^n)}(t)f)(s)},{e^*_{2^{k-1}+j}}\big\rb
\,dw(t)\Big|^p  \Big)^{\frac{q}{p}} \,ds\\
& \qquad \geq \int_{0}^{2^{-un}}\Big( \sum_{k=1}^{n} \sum_{j=0}^{2^{k-1}-1}
2^{-kr}\E|\Delta w_{I_{k,j}^n}|^p  \Big)^{\frac{q}{p}} \,ds\\
& \qquad = \int_{0}^{2^{-un}}\Big( \sum_{k=1}^{n} 2^{k-1}
2^{-kr}2^{-n\frac{p}{2}}\E|\gamma|^p  \Big)^{\frac{q}{p}} \,ds 
\\ & \qquad \geq 2^{-un-1}2^{n(1-r-\frac{p}{2})\frac{q}{p}}
(\E|\gamma|^p)^{\frac{q}{p}},
\end{aligned}$$
where in the second inequality we use \eqref{Snfl} 
and $\gamma$ denotes a standard Gaussian random variable. Thus if
$-u+(1-r-\frac{p}{2})\frac{q}{p}>0$, that is, if $q>up/(1-r-\frac{p}{2})$
(recall that
$r<1-\frac{p}{2}$), the left-hand side expression diverges as $n\rightarrow
\infty$.

\section*{Acknowledgments}
We thank Ben Goldys, Arnulf Jentzen, Markus Kunze, and
Mark Veraar for helpful comments and for suggesting several improvements. We
also thank the anonymous referee for pointing out various references.

\end{document}